\documentclass[12pt]{amsart}

\usepackage{latexsym, amssymb, amsmath,ulem,soul,esint}
\usepackage{tikz}

\usepackage{amsfonts, graphicx}
\usepackage{graphicx,color}
\newcommand{\be}{\begin{equation}}
\newcommand{\ee}{\end{equation}}
\newcommand{\beq}{\begin{eqnarray}}
\newcommand{\eeq}{\end{eqnarray}}

\newtheorem{thm}{Theorem}[section]

\newtheorem{lma}{Lemma}[section]
\newtheorem{prop}{Proposition}[section]
\newtheorem{cor}{Corollary}[section]

\usepackage{mathrsfs} 

\usepackage{wasysym,stmaryrd}
\newtheorem{claim}{Claim}[section]
\theoremstyle{remark}

\numberwithin{equation}{section}

\def\be{\begin{equation}}
\def\ee{\end{equation}}
\def\bee{\begin{equation*}}
\def\eee{\end{equation*}}

\def\K{K\"ahler }

\def\KR{K\"ahler-Ricci }
\def\Ric{\text{\rm Ric}}
\def\Rm{\text{\rm Rm}}

\def\e{\varepsilon}

\def\a{{\alpha}}
\def\b{{\beta}}

\begin{document}

\title[]
{Uniqueness of Ricci flow with scaling invariant estimates}

\author{Man-Chun Lee}
\address[Man-Chun Lee]{Department of Mathematics, The Chinese University of Hong Kong, Shatin, N.T., Hong Kong
}
\email{mclee@math.cuhk.edu.hk}

\renewcommand{\subjclassname}{
  \textup{2020} Mathematics Subject Classification}

\date{\today}

\begin{abstract}
In this work, we prove the uniqueness for complete non-compact Ricci flow with scaling invariant curvature bound. This generalizes the earlier work of Chen-Zhu, Kotschwar and covers most of the  examples of Ricci flows with unbounded curvature. In dimension three, we use it to show that complete Ricci flow starting from uniformly non-collapsed, non-negatively curved manifold is unique, extending the strong uniqueness Theorem of Chen. This is based on fixing gauge through Ricci-harmonic map heat flow in unbounded curvature background.
\end{abstract}


\maketitle

\markboth{Man-Chun Lee}{}

\section{Introduction}

Let $(M^n,g_0)$ be a smooth  Riemannian manifold. The Ricci flow starting from $g_0$ is a one parameter family of metrics $g(t),t\in [0,T]$ satisfying
\begin{equation}
\partial_t g(t)=-2\Ric(g(t)),\;\; g(0)=g_0.
\end{equation}
This evolution system was introduced by Hamilton in \cite{Hamilton1982} and now it has proved to be powerful in the research of differential geometry and lower dimensional topology.

In this work, we are primarily interested in the analytic nature of Ricci flow. As the Ricci flow is only weakly parabolic, the short-time existence and uniqueness of solutions do not follow directly from standard parabolic PDE theory. For compact manifolds $M$, Hamilton established these properties through his seminal work \cite{Hamilton1982}, employing a Nash-Moser-type inverse function theorem approach. DeTurck \cite{DeTurck1983} later significantly simplified this framework by introducing the strictly parabolic Ricci-DeTurck flow, which is related to the Ricci flow through a diffeomorphism gauge choice. The non-compact case presents substantially greater challenges. Under the assumption of bounded initial curvature, Shi \cite{Shi1989} constructed a bounded curvature solution to the Ricci flow. Uniqueness in this setting was subsequently resolved by Chen-Zhu \cite{ChenZhu2006}, who generalized Hamilton's methodology by coupling the Ricci flow with harmonic map heat flow. Kotschwar \cite{Kotschwar2014} later streamlined these uniqueness results using energy methods, providing a more direct proof.

Nevertheless, several existence results have been established for Ricci flows starting from complete initial metrics with potentially unbounded curvature. For a comprehensive overview of this active field, see Simon’s survey paper \cite{Simon2024}, which synthesizes key developments including those in \cite{CabezasWilking,BamlerCabezasWilking2019,
ChanLeeHuang2024,ChuLee,GiesenTopping2011,HochardPhd,
Lai2019,LeeTam2021,LeeTopping2022,SimonTopping2021}. Notably, most solutions constructed in these works exhibit curvature decay of the form $\a t^{-1}$ for some $\a>0$ as the flow evolves. This class of solutions is particularly significant for two reasons: First, their estimates are scaling-invariant, making them robust under parabolic dilations. Second, they naturally model the geometric smoothing of metric cones at infinity, a central feature of complete non-compact manifolds with Euclidean volume growth. These analytic and geometric properties have proven indispensable in applications of Ricci flow to non-compact geometries, underpinning many recent advances in the field.

Motivated by advances in Ricci flow theory for smooth non-compact initial data, we re-examine the uniqueness problem without imposing bounded curvature conditions. A fundamental challenge arises from the fact that uniqueness for parabolic equations on complete non-compact manifolds generally fails without growth restrictions on solutions. Our focus centers on Ricci flow solutions exhibiting scaling-invariant curvature decay, a natural class of flows preserving geometric information under parabolic rescaling. Within this framework, we establish the following central result:

\begin{thm}\label{thm:unique}
Suppose $(M,g_0)$ is a complete non-compact manifold. If $g(t)$ and $\tilde g(t)$ are complete  solutions to Ricci flow on $M\times [0,T]$ such that $g(0)=\tilde g(0)=g_0$ and 
\begin{equation}
|\Rm(g(t))|+|\Rm(\tilde g(t))|\leq \a t^{-1}
\end{equation}
on $M\times (0,T]$ for some $\a>0$, then $g(t)\equiv \tilde g(t)$ on $M\times [0,T]$.
\end{thm}

This result consequently applies to numerous Ricci flows constructed in the literature, including those developed in \cite{BamlerCabezasWilking2019,Lai2019,
LeeTam2021,SimonTopping2021,HochardPhd}, see also the reference therein. When the curvature is less than scaling invariant (i.e. bounded by $O(t^{-\gamma})$ for $\gamma\in (0,1)$), uniqueness was established by Kotschwar \cite{Kotschwar2014,Kotschwar2017}. When the flow are assumed to be equivalent to the initial metric $g_0$, uniqueness had also previously been demonstrated by the author and Ma \cite{LeeMa}.

The fundamental challenge in establishing Ricci flow uniqueness without bounded curvature assumptions stems from the absence of a canonical gauge-fixing mechanism. When the curvature is uniformly integrable as $t\to 0$, $g(t)$ is uniformly equivalent to $g(T)$ with bounded curvature. This way of gauge fixing enables us to compare metric freely. This geometric stability underpins Kotschwar's energy method approach \cite{Kotschwar2014}, where bounded curvature enables direct metric comparisons. To address the gauge problem under weaker  $\a/t$ type curvature decay, we revive Chen-Zhu's harmonic map heat flow strategy, originally devised for bounded curvature settings. Our key innovation lies in coupling two evolving Ricci flows $g(t),\tilde g(t)$  through a Ricci-harmonic map heat flow:
\begin{equation}
\partial_t F=\Delta_{g(t),\tilde g(t)}F.
\end{equation}
This system transforms one Ricci flow $g(t)$ into a Ricci-DeTurck flow relative to $\tilde g(t)$, leveraging the analytic advantages of strict parabolicity. Two principal obstacles arise in this framework: The target flow's scaling-invariant geometry precludes standard existence arguments; Spatial infinity introduces analytic subtleties absent in bounded curvature regimes.  To overcome this, we construct local solution with quantitative control using strongly the time zero asymptotic to get quantitative estimates. This relies on  the local maximum principle developed by the author and Tam \cite{LeeTam2022} and idea of Hochard \cite{HochardPhd} on building parabolic solutions with scaling invariant geometric estimates.

On the other hand, in a seminal contribution, Chen \cite{Chen2009} established strong uniqueness for complete Ricci flows emerging from Euclidean initial data in dimension three, leveraging delicate asymptotic properties of the flow. When the initial three-fold is only of $\Ric\geq 0$ and is uniformly volume non-collapsed, a short-time existence was established by Simon-Topping \cite{SimonTopping2021}. By synthesizing our new uniqueness theorem and curvature estimate from Perelman point-picking argument, we extend Chen’s uniqueness to non-Euclidean initial geometries which is uniformly volume non-collapsed and have $\Ric\geq 0$. We show that all complete Ricci flow solution must coincide with their constructed solution, for a short-time.

\begin{cor}\label{cor:strong-unique}
Suppose $(M^3,g_0)$ is a complete non-compact three dimensional manifold such that 
\begin{enumerate}
\item $\Ric(g_0)\geq 0$;
\item $\inf_M\mathrm{Vol}_{g_0}\left(B_{g_0}(x,1) \right)=v_0$ for some $v_0>0$,
\end{enumerate}
then any complete solution to the Ricci flow coincides with the solution constructed by  Simon-Topping \cite{SimonTopping2021} for a short-time.
\end{cor}

The paper is organized as follows. In Section~\ref{sec:pre-RF}, we collect some and prove some preliminaries estimate of Ricci flow, Ricci-DeTurck flow and Ricci-harmonic map heat flow. In section~\ref{sec:RHMF}, we construct local solution to the Ricci-harmonic map heat flow with quantitative estimate depending on scaling invariant curvature decay. In Section~\ref{sec:unique}, we use the quantitative local solution to prove uniqueness and strong uniqueness of Ricci flow.

\vskip0.2cm

{\it Acknowledgement}: The author would like to thank Shaochuang Huang and  Laura Bradby. The author  would like to thank  Felix Schulze and Peter Topping for favourable discussion which leads to a more transparent proof of Claim~\ref{claim:diff-local}. The author was partially supported by Hong Kong RGC grant (Early Career Scheme) of Hong Kong No. 24304222 and No. 14300623, a NSFC grant No. 12222122  and an Asian Young Scientist Fellowship.

\section{Some Preliminaries on Ricci flows}\label{sec:pre-RF}

In this section, we collect and prove several preliminary results for the Ricci flow under scaling invariant control. These results are purely local in nature and will play a crucial role in constructing global solutions to the Ricci–harmonic map heat flow.

\subsection{Distance distortion under $c/t$}

We begin by recording the distance distortion estimate for the Ricci flow under scaling invariant control. This estimate will be used repeatedly throughout this work.
\begin{lma}[Corollary 3.3 in \cite{SimonTopping2023}]
\label{lma:l-balls}
For $n\geq 2$ there exists a constant $\beta\geq 1$ depending only on $n$ such that the following is true. Suppose $(N^n,g(t))$ is a Ricci flow for $t\in [0,S]$ and $x_0\in N$   with $B_{g(0)}(x_0,r)\Subset N$ for some $r>0$, and $\Ric_{g(t)}\leq a/t$ on $B_{g_0}(x_0,r)$ for each $t\in (0,S]$. Then
$$B_{g(t)}\left(x_0,r-\beta\sqrt{a t}\right)\subset B_{g_0}(x_0,r).$$
\end{lma}

\subsection{Ricci-harmonic map heat flow}

Suppose $f:(M,g)\to (N,\tilde g)$ is a smooth map, then $df$ is a section of $T^*M\otimes f^{-1}(TN)$ where $f^{-1}(TN)$ is the pull-back bundle by $f$. Let $D$ be the covariant derivative induced by the Riemannian connections of $g$ and $\tilde g$. Then $D^k df$ is a section of $(T^*M)^{k+1}\otimes  f^{-1}(TN)$. The trace $\Delta_{g,\tilde g} f$ of $Ddf$ with respect to $g$ is called the tension field which is a vector field along $f$. If $g(t)$ is a smooth Ricci flow on $M$ and $\tilde g(t)$ is a smooth Ricci flow on $N$, then the Ricci-harmonic map heat flow is given by
\begin{equation}\label{eqn:RHMHF}
\partial_t F=\Delta_{g(t),\tilde g(t)}F
\end{equation}
where $F:M\times [0,T]\to N$ and $\Delta_{g,\tilde g}F$ at time $t$ is the tension field of $F(\cdot, t)$ with respect to the metric $g(t)$ and $\tilde g(t)$.

In what follows, both the metrics on domain and target  will be evolved under the Ricci flow, 
$$\partial_t g(t)=-2\Ric(g(t))\quad\text{and}\quad \partial_t \tilde g(t)=-2\Ric(\tilde g(t)).$$

 Since the bundle $(T^*M)^{k+1}\otimes  f^{-1}(TN)$ changes with the time,  following \cite{ChenZhu2006} we define a covariant time derivative $D_t$ as follows:
\begin{equation}
D_t u^\a_{i_1i_2...i_{k+1}}=\frac{\partial}{\partial t} u^\a_{i_1i_2...i_{k+1}}+\tilde\Gamma^\a_{\b\gamma} F^\b_t u^\gamma_{i_1i_2...i_{k+1}}
\end{equation}
for section $u$ in $ (T^*M)^{k+1}\otimes  F^{-1}(TN)$, where $\tilde\Gamma$ denotes the connection with respect to $\tilde g(t)$. We will omit the index $g(t), \tilde g(t)$ in $\Delta_{g,\tilde g}$ when the content is clear. We will always denote $n=\mathrm{dim}(M)$ and $m=\mathrm{dim}(N)$. When we refer to smooth map, we always use $\widetilde \Rm,\tilde\nabla$ to denote curvature and connection on the target.

We want to derive some a-priori estimates along the Ricci-harmonic map heat flow. We start with its evolution equations.
\begin{lma}\label{lma:evo-dF}
Along \eqref{eqn:RHMHF}, $F$ satisfies 
\begin{equation*} 
\begin{split}
(D_t-\Delta)D^k dF&=\sum_{\ell=0}^k D^\ell\left( \Rm + \widetilde\Rm* dF* dF\right)* D^{k-\ell}dF\\
&\quad +\sum_{\ell=0}^{k-1} D^\ell\left(\nabla \Rm+\tilde\nabla \widetilde\Rm*dF+\widetilde\Rm*dF*DdF \right)*D^{k-1-\ell}dF
\end{split}
\end{equation*}
where the second sum is understood to be $0$ for $k=0$. In particular for any $k\in \mathbb{N}_{\geq 0}$, we have
\begin{equation*}
\begin{split}
\left(\frac{\partial}{\partial t}-\Delta\right) |D^k dF|^2&\leq-2|D^{k+1} dF|^2+ C(n,m,k)|D^k dF|  \Biggl[\sum_{\ell=0}^k|\nabla^\ell \Rm||D^{k-\ell}dF|\\
&\quad +\sum_{\ell=0}^k\sum_{p=0}^{\ell}\sum_{i+j=p} |D^{\ell-p}\widetilde\Rm | |D^{i}dF||D^j dF||D^{k-\ell}dF|\\
&\quad+\sum_{\ell=0}^{k-1}\sum_{p=0}^{\ell} |D^{\ell-p}\tilde\nabla\widetilde\Rm | |D^pdF| |D^{k-1-\ell} dF|
\Biggr]+C|\widetilde\Rm||D^kdF|^2.
\end{split}
\end{equation*}
\end{lma}
\begin{proof}

The second inequality follows from the first and the Ricci flow equations.  It remains to verify the first equation. This is almost identical to \cite[Lemma 2.10]{ChenZhu2006}. We sketch the proof for readers' convenience. When $k=0$, we use Ricci identity to obtain
\begin{equation}
\left( D_t -\Delta_{g,\tilde g}\right) F^\a_i=-R_i^pF_p^\a+g^{jk}\tilde R_{\b\gamma\delta}\,^\a F^\b_k F_i^\gamma F^\delta_j
\end{equation}
This proves the case $k=0$. Suppose this holds for $k$, we want to show the case of $k+1$. This follows from Ricci identities that { 
\begin{equation}
\begin{split}
(D_t-\Delta)(Du)= D\left((D_t-\Delta)u\right)+A*Du+(DA+B)*u, 
\end{split}
\end{equation}
where
\begin{equation}
\left\{
\begin{array}{ll}
A:=\mathrm{Rm}+\widetilde\Rm*dF*dF;\\[1mm]
B:=\nabla \Rm+\tilde\nabla \widetilde\Rm*dF+\widetilde\Rm*dF*DdF,
\end{array}
\right.
\end{equation}}
for any section $u$ in $(T^*M)^\ell \otimes F^{-1}(TN)$, $\ell\geq 1$. Here we have used Ricci flow equation to obtain
\begin{equation}
\partial_t \Gamma_{ij}^k=-g^{kl}\left(\nabla_i R_{jl}+\nabla_j R_{il}-\nabla_l R_{ij} \right)
\end{equation}
for both $g(t)$ and $\tilde g(t)$.
\end{proof}

When $k=0$, we have a sharper evolution equation which might be of independent interest, see also \cite[Theorem 1.1]{BamlerBrendle}.
\begin{lma}\label{lma:evo-energy}
Along \eqref{eqn:RHMHF}, $F$ satisfies 
\begin{equation}
\left(\frac{\partial}{\partial t}-\Delta \right)|dF|^2=-2|DdF|^2+2g^{ij} \left[g^{kl}(F^*\widetilde\Rm)_{iklj}-(F^*\widetilde \Ric)_{ij}\right]
\end{equation}
\end{lma}
\begin{proof}
This follows from Ricci flow equation and direct computation using Ricci identity.
\end{proof}

In \cite[Lemma 2.12]{ChenZhu2006}, a simpler version of Lemma~\ref{lma:evo-dF} was used to show that, for the harmonic map heat flow, one can improve the regularity of $F$ by assuming a $C^2$ estimate along the flow together with bounded curvature. The following proposition shows that similar estimates also hold locally under only scaling invariant a‑priori assumptions.

\begin{prop}\label{prop:HF-boot}
Suppose $(M^n,g(t)), (N^m,\tilde g(t)),t\in [0,T]$ are two smooth { complete} solutions to the Ricci flow and $x_0\in M$ such that for some $\rho>0$
\begin{enumerate}
\item $B_{g(0)}(x_0,4\rho)\Subset M$;
\item $|\Rm(g(t))|+|\Rm(\tilde g(t))|\leq \a t^{-1}$ for $t\in(0,T]$, for some $\a>0$.
\end{enumerate}
If $F:B_{g(0)}(x_0,4\rho)\times [0,T]\to  N$ is a smooth solution to \eqref{eqn:RHMHF} such that $|d F|^2\leq \Lambda_0$ for some $ \Lambda_0>0$, then for all $k\in \mathbb{N}$, there exists $L_k(n,m,\a,\Lambda_0,k)$, $S_k(n,m,\a, \Lambda_0,k)>0$ so that for all $(x,t)\in B_{g(0)}(x_0,\rho)\times (0,T\wedge S_k\rho^2]$ and $ 0\leq \ell\leq k$,
\begin{equation}
 |D^\ell dF|^2\leq L_k t^{-\ell}. 
\end{equation}
\end{prop}
\begin{proof}
By scaling, we assume $\rho=1$.
This follows from a slight modification of  the argument in \cite[Lemma 2.12]{ChenZhu2006} except the curvature bound is bounded by $\a t^{-1}$ instead. We include the proof for readers' convenience. We will use $C_i$ to denote any constant depending only on $n,m,\a, \Lambda_0,k$.

We prove it by induction. Clearly, the case $k=0$ of the Proposition holds by assumption. Suppose the proposition holds for some $k$. Let $x_1\in B_{g(0)}(x_0,1)$ so that $B_{g(0)}(x_1,1)\Subset M$. We consider the rescaled solution $h(t)=4^2 g(4^{-2} t)$, $\tilde h(t)=4^2 \tilde g(4^{-2}t)$ and $\hat  F(t)=F(4^{-2}t)$ for $t\in [0,4^2T]$ so that $B_{h(0)}(x_1,4)\Subset M$. By induction, there is $L_k,S_k>0$ such that for all $0\leq \ell\leq k$, 
\begin{equation}
|D^\ell d\hat F|^2\leq L_k t^{-\ell}
\end{equation}
on $B_{h(0)}(x_1,1)\times (0,16T\wedge S_k]$. We will omit hat for notation convenience. We can as well assume by Shi's estimate \cite{Shi1989} that for all $0\leq \ell\leq k$,
\begin{equation}
|\nabla^\ell\Rm|^2+|\tilde\nabla^\ell \widetilde\Rm|^2\leq L_k t^{-\ell-2}
\end{equation}
on $B_{h(0)}(x_1,1)\times (0,16T\wedge S_k]$.

It then follows from Lemma~\ref{lma:evo-dF}, { Shi's estimates and induction} that 
\begin{equation}
\left\{
\begin{array}{ll}
\displaystyle\left(\frac{\partial}{\partial t}-\Delta\right) |D^{k}  d F|^2&\leq -2|D^{k+1}dF|^2 +C_0t^{-1-k};\\[2mm]
\displaystyle\left(\frac{\partial}{\partial t}-\Delta\right) |D^{k+1} dF|^2&\leq -2|D^{k+2}dF|^2 +C_0|D^{k+1}dF|^2 t^{-1}+C_0t^{-2-k}.
\end{array} 
\right.
\end{equation}

By \cite[Lemma 8.3]{Perelman2002}, for some large $\Lambda_n>0$ the function $\eta(x,t)=d_{h(t)}(x,x_1)+\Lambda_n\sqrt{\a t}$ satisfies $(\partial_t-\Delta_{h(t)}) \eta\geq 0$ in the sense of barrier whenever $d_{h(t)}(x,x_1)\geq \sqrt{\a t}$. Fix a smooth non-increasing function $\phi$ on $[0,+\infty)$ such that $\phi=1$ on $[0,\frac12]$, vanishes outside $[0,1]$ and satisfies $\phi''\geq -10^3\phi$, $|\phi'|^2\leq 10^3 \phi$. Define a cut-off function $\Phi(x,t)=e^{-10^3t}\phi(\eta(x,t))$. We can assume $\Phi$ is compactly supported in $B_{h(0)}(x_1,1)\times [0,16T\wedge S_k]$ by Lemma~\ref{lma:l-balls}. We might assume $ (\partial_t-\Delta_{h(t)})\Phi\leq 0$ in the sense of barrier, by shrinking $S_k$.  We might also assume $\Phi$ to be smooth when applying maximum principle.

Let $\Lambda$ be a large constant to be chosen. We now consider the test function
$$G:=\Lambda t^{k} |D^k dF|^2+ t^{k+1}\Phi |D^{k+1}dF|^2$$
on the support of $\Phi$. We only need to examine its maximum on $B_{h(0)}(x_1,1)\times [0,16T\wedge S_k]$. It suffices to consider the case when the maximum is attained in the interior $(z_0,t_0)$ in which
\begin{equation}
\begin{split}
0&\leq (\partial_t-\Delta_{h(t)})G|_{(z_0,t_0)}\\
&\leq k\Lambda t^{k-1} |D^k dF|^2+ \Lambda t^k \left( -2|D^{k+1}dF|^2+C_0t^{-1-k}\right)\\
&\quad + (k+1)t^k \Phi |D^{k+1}dF|^2 -2t^{k+1} \langle \nabla\Phi,\nabla |D^{k+1}dF|^2\rangle \\
&\quad +t^{k+1}\Phi \left(-2|D^{k+2}dF|^2+C_0t^{-1}|D^{k+1}dF|^2+C_0t^{-2-k} \right)  \\
&\leq -t^{k+1}\Phi |D^{k+2}dF|^2 + t^k|D^{k+1}dF|^2 \left(-2\Lambda +(k+1)+ C_0 \right)\\
&\quad +C_1\Lambda t^{-1}+ C_1 t^{k+1} \frac{|\nabla\Phi|^2}{\Phi} |D^{k+1}dF| ^2.
\end{split}
\end{equation}

Since by construction, $ |\nabla\Phi|^2\leq 10^3\Phi$. If we choose $\Lambda$ large enough, then we conclude that at $(z_0,t_0)$,
\begin{equation}
\begin{split}
  t^{k+1}|D^{k+1}dF|^2 |_{(z_0,t_0)} \leq C_2
\end{split}
\end{equation}
By maximum principle, $G(x_1,t)\leq C_3$ for $t\in [0,16T\wedge S_k]$. Since $\Phi(x_1,t)\geq \frac12$ if $S_k$ is small enough, we obtain estimate of $|D^{k+1}dF|(x_1,t)$. By rescaling it back to the original Ricci-harmonic map heat flow, we conclude that 
\begin{equation}
|D^{k+1}dF( x_1,t)|^2\leq C_4 t^{-k-1}
\end{equation}
for $t\in (0,T\wedge 16^{-1}S_k]$. This completes the proof by induction.
\end{proof}

\medskip
\subsection{Existence in bounded curvature}

In \cite[Theorem 2.1]{ChenZhu2006}, Chen–Zhu established the existence of the harmonic map heat flow coupled with a bounded-curvature Ricci flow, starting from the identity map, when the target manifold is equipped with a fixed metric of bounded curvature. Following their method, it is straightforward to see that the same conclusion holds for the harmonic map heat flow coupled with evolving metrics of quasi‑bounded geometry; see also \cite{HuangTam2022}. For our purposes, we require only the following weaker version.

\begin{thm}\label{thm:chenzhu-harmonic-map}
For any $n\in \mathbb{N}, \Lambda_0>0$, there exists $\hat T_1(n,\Lambda_0)>0$ such that the following holds: Suppose $(M^n,h(t))$  and $(N^n,\tilde h(t))$ for $t\in [0,T]$ are two smooth family of complete metrics such that for some $\rho>0$ and for all $t\in [0,T]$,
\begin{enumerate}
\item $|\Rm(\tilde h)|+|\Rm(h)|\leq \rho^{-2}$;
\item $|\partial_t h|+|\partial_t \tilde h|\leq \rho^{-2}$ and;
\item $\sup_{M} |\partial^i_t \nabla^j \Rm(h)|+\sup_{N} |\partial_t^i \tilde \nabla^j \Rm(\tilde h)|<\infty$ for all $i,j\in \mathbb{N}$.
\end{enumerate}
If $f:M\to N$ is a smooth map such that  
\begin{enumerate}
\item[(i)] $|d f|^2_{h(0),\tilde h(0)}\leq \Lambda_0$;
\item[(ii)] $\sup_M|Ddf|^2<+\infty$;
\item[(iii)] $f(M)\Subset N$.
\end{enumerate}
 Then the harmonic map heat flow:
\begin{equation*}
\left\{
\begin{array}{ll}
\partial_t F(x,t)=\Delta_{h(t),\tilde h(t)} F(x,t);\\[2mm]
F(x,0)=f(x)
\end{array}
\right.
\end{equation*}
has   a solution on $M\times [0, T\wedge \hat T_1\rho^2]$ such that $|d F|^2(x,t)\leq 10\Lambda_0$.
\end{thm}
\begin{proof}
By scaling, we might assume $\rho=1$. This is almost identical to the proof of \cite[Theorem 2.7]{ChenZhu2006}, we only sketch the proof and point out the necessary modifications. We choose a open set $V\Subset N$ such that $f(M)\Subset V$ and $\mathrm{inj}(\tilde h(t))>\iota>0$ on $V\times [0,T]$ for some $\iota>0$. Let $\{\Omega_i\}_{i=1}^\infty$ be a compact exhaustion of $M$. Following the proof of \cite[Lemma 2.9]{ChenZhu2006},  we construct a sequence of harmonic map heat flow $F_i$ coupled with $h(t)$ and $\tilde h(t)$ on each $\Omega_i\times [0,T_i]$ with Dirichlet boundary condition $F=f$. Using $\mathrm{inj}(\tilde h)>\iota$ and curvature estimates of $h$ and $\tilde h$, the sequence of Dirichlet solution $F_i$ satisfies $ F_i(\Omega_i,t)\Subset V$ for $t\in [0,T_i]$ and uniform local estimates, see the proof of \cite[Lemma 2.11 \& 2.12]{ChenZhu2006}. In particular, $T_i>T_0>0$ for some $T_0>0$ possibly depending also on $V$. The higher order estimates also enable us to construct global solution by limiting argument with bounded global energy $\sup_{M\times [0,T_0]}|dF|<+\infty$. We now show that the constructed solution can be extended to some uniform time depending only on upper bound of $|\Rm(h)|,|\Rm(\tilde h)|, |\partial_t h|,|\partial_t \tilde h|$ and $\Lambda_0$.  It suffices to control the energy uniformly for uniform short time. Using the estimate on curvature and time derivatives of metrics, we have 
\begin{equation}
\left(\frac{\partial}{\partial t}-\Delta_{h(t)}\right) |dF|^2 \leq C_n (|dF|^4+|dF|^2).
\end{equation}
Since the energy is bounded for short-time, we apply maximum principle to $|dF|$ to show that for some $\hat T_1(n,\Lambda_0)>0$, we must have $|dF|^2\leq 10\Lambda_0$ if $t<\hat T_1\wedge T$.  Thus the maximal existence time must be beyond $\hat T_1(n,\Lambda_0)\wedge T>0$.  This completes the proof.
\end{proof}

Instead of using a Dirichlet exhaustion, we construct local approximations by exhausting $M$ with complete manifolds. This can be achieved by modifying the incomplete metric near its ends. In this way, we minimize the required regularity assumptions.

\begin{lma}\label{lma:conformal}
There exists $\sigma(n)\in (0,1),\Lambda_1(n)>1$ such that the following holds:  Suppose $g(t),t\in [0,T]$ is a family of smooth metrics on $N$ such that 
\begin{enumerate}
\item $|\partial_t g(t)|+|\Rm(g(t))|\leq \rho^{-2}$; 
\item $|\nabla \partial_t g(t)|\leq \rho^{-3}$
\end{enumerate}
for some $\rho>0$. Then there exists a smooth family of complete metrics $h(t)$ on $N\times [0,T\wedge \sigma \rho^2]$  such that
\begin{enumerate}
\item[(i)] $h(t)\equiv g(t)$ on $N_\rho:=\{x\in N: B_{g(0)}(x,\rho)\Subset N\}$;
\item [(ii)] $|\Rm(h(t))|+|\partial_t h(t)|\leq \Lambda_1 \rho^{-2}$ on $N\times [0,T\wedge \sigma\rho^2]$;
\item[(iii)] $h(t)\geq g(t)$ on $N\times [0,T\wedge \sigma\rho^2]$.
\end{enumerate}
Furthermore, if $g(t),t\in [0,T]$ is smooth metric on a open set $U$ such that $N\Subset U$, then $h(t)$ has bounded geometry of infinity  order.
\end{lma}
\begin{proof}
This follows from a modification of \cite[Corollary IV.1.2.]{HochardPhd}, in which we now allow the metrics to be evolving. By scaling, we assume $\rho=1$. On $N$, we let 
$$\rho_1(x):=\sup\{{ 1>} r>0: B_{g(0)}(x,r)\Subset N, \sup_{B_{g(0)}(x,r)}|\Rm(g(0))|\leq r^{-2}\}.$$

By \cite[Lemma IV.1.3]{HochardPhd}, we can mollify $\frac12 \rho_1$ to a smooth function $\tilde\rho$ on $N$ such that 
\begin{equation}\label{eqn:rho_1-bdd}
\frac12 \rho_1(x)\leq \tilde\rho(x)\leq \frac32 \rho_1(x),\;\; |\nabla^{g(0)}\tilde \rho|+\rho_1|\nabla^{2,g(0)}\tilde\rho|\leq C_n
\end{equation}
on $N$, for some $C_n>0$.

Given $\e>0$, we let $f:(0,+\infty)\to [0,+\infty)$ be a smooth function (smoothed from $-\log\left( 1- (1-\e^{-1}x)^2\right)$) such that 
\begin{enumerate}
\item $f(x)=-\log\left( 1- (1-\e^{-1}x)^2\right)$ for $x\leq 0.9\e$;
\item $f(x)=0$ for $x\geq 1.1\e$;
\item $0\geq f'(x)\geq  -2\e^{-1}$, for $(0.9\e, 1.1\e)$;
\item $0\leq f''(x)\leq   8\e^{-2}$, for $x\in (0.9\e, 1.1\e)$.
\end{enumerate}
We now define the family of metric $h(t):= e^{2f\circ \tilde \rho}g(t)$ on $M$ so that $h(t)$ is complete metric on $N$ for all $t>0$, $h(t)=g(t)$ on $\{x\in N: \rho_1(x)\geq \frac{22}{10}\e\}$. When $g(t)$ is static, it was shown by Hochard \cite{HochardPhd} using direct computation that the curvature of $h(t)$ is bounded by $ C_n\e^{-2}$ for some $C_n>0$. If $g(t)$ is time-independent, the result follows by fixing a small dimensional $\e>0$. To extend it to evolving family of metrics, it suffices to show that $\nabla^{g(0)}$ in \eqref{eqn:rho_1-bdd} can be swapped to $\nabla^{g(t)}$.

Since $|\partial_t g(t)|\leq 1$, we see that for all $(x,t)\in  N\times [0,T\wedge \log 2]$,
\begin{equation}
\frac12 g(0)\leq g(t)\leq 2g(0)
\end{equation}
and thus $|\nabla^{g(t)}\tilde \rho|\leq C_n$. For higher order, we use the assumption $|\nabla\partial_t g|\leq 1$ to see that $\Psi=\Gamma(g(t))-\Gamma(g(0))$ satisfies 
\begin{equation}
\partial_t |\Psi|\leq C_n|\Psi||\partial_t g|+|\nabla\partial_tg|.
\end{equation}
Integration shows that $|\Psi|$ is uniformly bounded and hence,
\begin{equation}
|\nabla^{2,g(t)}\tilde\rho|\leq |\nabla^{2,g(0)}\tilde\rho|+|\Psi* \partial \tilde\rho|\leq C_n\rho_1^{-1}
\end{equation}
when $\rho_1(x)\leq \frac{22}{10}\e$. Using this, the bound of $|\Rm(h(t))|$ follows from direct computation.

Furthermore, if $g(t)$ is smooth on a open set $U$ where $N\Subset U$, then $h(0)$ has bounded geometry by argument in \cite[Lemma 4.3]{LeeTamJDG}. Since the flow is smooth, the same holds for $h(t)$.
\end{proof}

As an immediate consequence of Lemma~\ref{lma:conformal} together with Theorem~\ref{thm:chenzhu-harmonic-map}, we have a local existence for Ricci-harmonic map heat flow.
\begin{prop}\label{prop:extension-HM}
For any $\Lambda_0>0$, there exists $\tilde T_2(n,\Lambda_0)>0$ such that the following is true: Suppose $g(t),t\in [0,T]$ is a smooth solution to Ricci flow on a open set $ U\Subset M$ (not necessarily complete), and $\tilde g(t),t\in [0,T]$ is a smooth complete solution to Ricci flow on $M'$ so that for $t\in [0,T]$,
\begin{enumerate}
\item $\sup_U|\Rm(g(t))|+\sup_{M'}|\Rm(\tilde g(t))|\leq \rho^{-2}$;
\item $\sup_U|\nabla\Rm(g(t))|\leq \rho^{-3}$;
\item $\sup_{M'}|\tilde\nabla^k\Rm(\tilde g(t))|<+\infty$, for all $k\geq 0$,
\end{enumerate}
for some $\rho>0$.
If $ f:\bar U\to M'$ is a smooth map such that $|df|^2_{g(0),\tilde g(0)}\leq \Lambda_0$, then there exists a one parameter family of smooth map $F$ on $U\times [0,T\wedge \tilde T_2\rho^2]$ such that $F(0)=f$ and 
\begin{enumerate}
\item[(a)] $\partial_t F=\Delta_{g(t),\tilde g(t)}F$;
\item[(b)] $|dF|^2 \leq 10\Lambda_0$
\end{enumerate} 
on $U_\rho:=\{x\in U: B_{g(0)}(x,\rho)\Subset U\}\times [0,T\wedge \tilde T_2\rho^2]$.
\end{prop}
\begin{proof}

By rescaling, we assume $\rho=1$. We apply Lemma~\ref{lma:conformal} to $g(t)$ with $N$ chosen to be $U$. Then there is a smooth family of complete metric $h(t)$ on $N\times [0,T\wedge \sigma]$ so that $h(t)\equiv g(t)$ on $N_1$ and $|\Rm(h(t))|+|\partial_t h(t)|\leq \Lambda_1$ on $N\times [0,T\wedge \sigma]$. Moreover, $h(t)$ has bounded geometry of infinity order. Here $\sigma=\sigma(n)$ and $\Lambda_1=\Lambda_1(n)$.

We now consider the smooth map $f:\left(N,h(0)\right)\to \left( M',\tilde g(0)\right)$ which satisfies $f(N)\Subset M'$, $\sup_N|Ddf|_{h(0),\tilde g(0)}<+\infty$ (by the construction of $h(0)$) and
\begin{equation}
|df|^2_{h(0),\tilde g(0)}\leq |df|^2_{g(0),\tilde g(0)}\leq \Lambda_0.
\end{equation}

We now apply Theorem~\ref{thm:chenzhu-harmonic-map} to obtain $\hat T_1(n,\Lambda_0)>0$ and a solution to the harmonic map heat flow  $F$ coupled with $h(t),\tilde g(t)$ on $U\times [0,T\wedge (\hat T_1 \Lambda_1^{-1}) \wedge \sigma]$ with $F(0)=f$ and $|dF|^2\leq 10\Lambda_0$. As $h(t)\equiv g(t)$ on $U_1:=\{ x\in U: B_{g(0)}(x,1)\Subset U\}$, this completes the proof by taking $\tilde T_2(n,\Lambda_0)=\sigma \wedge (\hat T_1\Lambda_1^{-1})$.
\end{proof}

\subsection{Ricci-DeTurck flow and a-priori estimates}\label{sec:RDF-intro}

In this sub-section, we discuss the relationship between Ricci-harmonic map heat flow and Ricci-DeTurck flow. Indeed given a solution $F$ to \eqref{eqn:RHMHF}, if $F$ is a diffeomorphism, then the one parameter family of metrics $\hat g(t):=(F_t^{-1})^* g(t)$ solves 
\begin{equation}\label{eqn:RDF-1}
\left\{
\begin{array}{ll}
\partial_t \hat g_{ij}=-2\Ric(\hat g)_{ij} +\nabla^{\hat g}_i V_j+\nabla^{\hat g}_j V_i;\\[2mm]
V^k=\hat g^{ij} \left[\Gamma^k_{ij}(\hat g)-\Gamma^k_{ij}(\tilde g) \right].
\end{array}
\right.
\end{equation} 
This is usually referred as Ricci-DeTurck flow with respect to $\tilde g(t)$ which is a strictly parabolic system. Furthermore, we see that 
\begin{equation}
V^k|_{(F_t(x),t)}=-\Delta_{g(t),\tilde g(t)} F^k|_{(x,t)}
\end{equation}
see \cite[Lemma 3.18]{ChowBookI}, so that we simplify the \eqref{eqn:RHMHF} as $\partial_t F_t(x)=-V\left(F_t(x),t \right)$.

Instead of \eqref{eqn:RDF-1}, we will alternatively work on the following equivalent equation:
\begin{equation}\label{eqn:RDF-2}
\begin{split}
\partial_t \hat g_{ij}&=\hat g^{pq} \tilde \nabla_p\tilde\nabla_q \hat g_{ij}-\hat g^{kl}\hat g_{ip}\tilde g^{pq} \tilde R_{jkql}-\hat g^{kl}\hat g_{jp} \tilde g^{pq}\tilde R_{ikql}\\
&\quad +\frac12 \hat g^{kl}\hat g^{pq}\big(\tilde \nabla_i \hat g_{pk}\tilde \nabla_j\hat  g_{ql}+2\tilde \nabla_k \hat g_{jp}\tilde \nabla_q \hat g_{il}-2\tilde \nabla_k\hat  g_{jp}\tilde \nabla_l \hat g_{iq}\\[1mm]
&\quad -2\tilde \nabla_j\hat  g_{pk}\tilde \nabla_l \hat g_{iq}-2\tilde \nabla_i\hat  g_{pk}\tilde \nabla_l \hat g_{qj}\big),
\end{split}
\end{equation}
 see \cite[Lemma 2.1]{Shi1989}.
We will work on \eqref{eqn:RDF-2} in this section.

\begin{lma}\label{lma:improved-evo-h}
There exists $ 10^{-2}>\e_1(n)>0$ such that the following holds:  Suppose $(N,\tilde g(t))$ is a smooth solution to the Ricci flow  and $\hat g(t)$ is a smooth solution to the Ricci-DeTurck flow with respect to $\tilde g(t)$ on $N^n\times [0,T]$ so that 
$$(1-\e_1)\tilde g(t)\leq \hat g(t)\leq (1+\e_1)\tilde g(t)$$
on $N\times [0,T]$. Then $h:=\hat g-\tilde g$ satisfies 
\begin{equation}
\left\{
\begin{array}{ll}
\displaystyle\left( \frac{\partial}{\partial t}-\hat g^{ij}\tilde\nabla_i \tilde\nabla_j\right) |h|^2 \leq -|\tilde\nabla h|^2+C_n |\Rm(\tilde g)|\cdot  |h|^2;\\[2mm]
\displaystyle\left( \frac{\partial}{\partial t}-\Delta_{\tilde g(t)}\right) |h|^2 \leq C_n \left(|\Rm(\tilde g)| + |\tilde\nabla^2 h|\right) |h|^2.
\end{array}
\right.
\end{equation}
\end{lma}
\begin{proof}

Recall that $\hat g$ satisfies \eqref{eqn:RDF-2}. By combining this with $\partial_t \tilde g(t)=-2\Ric(\tilde g(t))$, we see that $h(t)=\hat g(t)-\tilde g(t)$ satisfies 
\begin{equation}\label{eqn:RDF-3}
\begin{split}
\left(\frac{\partial}{\partial t}-\hat g^{ij}\tilde\nabla_i\tilde\nabla_j\right) h&=\hat g^{-1}* h*\hat g* \widetilde\Rm  + \hat g^{-1}*\hat g^{-1} * \tilde\nabla h *\tilde\nabla h,
\end{split}
\end{equation}
and thus 
\begin{equation}
\begin{split}
\left(\frac{\partial}{\partial t}-\hat g^{ij}\tilde\nabla_i\tilde\nabla_j\right)|h|^2&\leq  \left(-2+C_n\e_1\right)|\tilde\nabla h|^2+ C_n|\widetilde\Rm| |h|^2
\end{split}
\end{equation}
for some $C_n>0$. The assertion follows by choosing $\e_1$ sufficiently small. The second inequality is similar by using 
$$\hat g^{ij}-\tilde g^{ij}=-\hat g^{iq}\tilde g^{pj}h_{pq}.$$ 
This completes the proof.
\end{proof}

We now establish higher order estimate of $h$. When $\tilde g(t)$ has bounded curvature, this is standard, see \cite{Shi1989,Simon2002}. We show that the same conclusion holds when the curvature of $\tilde g(t)$ satisfies only scaling invariant bounds.

\begin{lma}\label{lma:RDF-higher-ord}
There exists $ 10^{-2}>\e_2(n)>0$ such that the following holds:  Let $(N,\tilde g(t))$ be a smooth solution to the Ricci flow (not necessarily complete) on $[0,T]$ and $x_0\in N$ such that 
\begin{enumerate}
\item $B_{\tilde g(0)}(x_0,4)\Subset N$;
\item $|\Rm(\tilde g(t))|\leq \a t^{-1}$ on $N\times (0,T]$. 
\end{enumerate}
If $\hat g(t)$ is a smooth solution to the Ricci-DeTurck flow with respect to $\tilde g(t)$ on $ N^n\times [0,T]$ so that 
$$(1-\e_2)\tilde g(t)\leq \hat g(t)\leq (1+\e_2)\tilde g(t)$$
on $N\times [0,T]$. Then for all $k\in\mathbb{N}$, there exists $\hat L_k(n,\a,k),\hat S_k(n,\a,k)>0$ such that for all $(x,t)\in B_{\tilde g(0)}(x_0,1)\times (0,T\wedge  \hat S_k]$,
\begin{equation*}
|\tilde\nabla^k \hat g(x,t)|\leq \hat L_k t^{-k/2}.
\end{equation*}
\end{lma}
\begin{proof}
We will use $C_i$ to denote any constants depending only on $n,\a,k$  and $L_i$ for dimensional constants.  We will choose $\e_2<\e_1$ where $\e_1$ is the constant from  Lemma~\ref{lma:improved-evo-h}.

We only focus on the case $k=1$. This is the hardest case which uses the smallness of $\e_1$. We might assume by Shi's estimate \cite{Shi1989} that 
\begin{equation}\label{eqn:Shi-si}
|\tilde\nabla \widetilde\Rm|\leq C_0t^{-3/2}
\end{equation}
on $ B_{\tilde g(0)}(x_0,3)\times (0,T\wedge 1]$. By differentiating 
\eqref{eqn:RDF-3}, we have for $k\geq 1$
\begin{equation}
\begin{split}
\left(\frac{\partial}{\partial t}-\hat g^{ij}\tilde\nabla_i\tilde\nabla_j\right) |\tilde\nabla^k h|^2&\leq   -\frac32|\tilde\nabla^{k+1}h|^2+ C_1 |\tilde\nabla^k h|\cdot \sum_{i+j+p+q=k} |\tilde\nabla^i \hat g^{-1}||\tilde\nabla^j h| |\tilde\nabla^p \hat g| |\tilde\nabla^q\widetilde\Rm|\\
&\quad +  L_1|\tilde\nabla^k h|\cdot \sum_{i+j+p+q=k}|\tilde\nabla^i \hat g^{-1}||\tilde\nabla^j \hat g^{-1}||\tilde\nabla^{p+1} h||\tilde\nabla^{q+1} h|.
\end{split}
\end{equation}

Specifying to case of $k=1$ becomes
\begin{equation}
\begin{split}
\left(\frac{\partial}{\partial t}-\hat g^{ij}\tilde\nabla_i\tilde\nabla_j\right) |\tilde\nabla h|^2
&\leq -\frac32|\tilde\nabla^2 h|^2+C_2 |\tilde\nabla h|^2| \widetilde\Rm|+ C_2|\tilde\nabla h||h||\tilde\nabla \widetilde\Rm|\\
&\quad + L_2|\tilde\nabla h|^4+L_2|\tilde\nabla h|^2|\tilde\nabla^2 h|\\
&\leq {-\frac54|\tilde\nabla^2 h|^2}+ L_3|\tilde\nabla h|^4 +C_3 t^{-2}
\end{split}
\end{equation}
on $ B_{\tilde g(0)}(x_0,3)\times (0,T\wedge 1]$. If we define
$G:= |\tilde\nabla h|^2 \left(1+\Lambda_0|h|^2 \right)$
where $\Lambda_0$ is a large constant to be chosen. Then on  $ B_{\tilde g(0)}(x_0,3)\times (0,T\wedge 1]$, $G$ satisfies 
\begin{equation}
\begin{split}
\left(\frac{\partial}{\partial t}-\hat g^{ij}\tilde\nabla_i\tilde\nabla_j\right)G
&\leq \Lambda_0|\tilde\nabla h|^2 \left(-|\tilde\nabla h|^2+C_4t^{-1} \right) -2\Lambda_0\langle \tilde\nabla |\tilde\nabla h|^2,\tilde\nabla |h|^2\rangle\\
&\quad  + (1+\Lambda_0 |h|^2)\left( - |\tilde\nabla^2 h|^2+L_3|\tilde\nabla h|^4+C_3t^{-2}\right)\\
&\leq  \left[-\frac12\Lambda_0+L_3(1+{\Lambda_0^2}\e_2^2) \right]|\tilde \nabla h|^4\\
&\quad - |\tilde\nabla^2 h|^2+(1+\Lambda_0)C_5 t^{-2}
\end{split}
\end{equation}
where we have used Lemma~\ref{lma:improved-evo-h} and $\e_2<\e_1$.  We choose $ \e_2(n):= \min\{1,\Lambda_0^{-1}\}$ and $\Lambda_0(n):=4(L_3+1)$ so that this is reduced to 
\begin{equation}
\left(\frac{\partial}{\partial t}-\hat g^{ij}\tilde\nabla_i\tilde\nabla_j\right)G
\leq -|\tilde\nabla h|^4 +C_6 t^{-2}\leq C_6t^{-2}- \frac{G^2}{(1+\Lambda_0 )^2}.
\end{equation}

Fix $x_1\in B_{\tilde g(0)}(x_0,1)$, we might assume $ B_{\tilde g(t)}(x_1,1)\Subset B_{\tilde g(0)}(x_0,3)$ for $t\in [0,S_1]$, if we shrink $S_1$,  by Lemma~\ref{lma:l-balls}. We claim that if $S_1$ is small enough, then the desired estimates hold. For $t\in [\frac12 t_1,t_1]$ where $t_1\in (0,S_1]$, we rescale the solution as follows: $\check h(t)=2t_1^{-1}h(\frac12 t_1+\frac12 t_1t)$ and $ \check g(t)=2t_1^{-1}\tilde g(\frac12 t_1+\frac12 t_1t)$ on $ B_{\check g(0)}(x_1, t^{-1/2})\times [0,1]$. 

Thanks to the scaling invariant curvature bound, we have 
\begin{equation}\label{eqn:scaled-control}
|\Rm(\check g(t))|+|\nabla \Rm(\check g(t))|\leq C_7
\end{equation}
$B_{\check g(0)}(x_1,t_1^{-1/2})\times [0,1]$. We take the distance function $d_{\check g(0)}(x,x_1)$ which satisfies 
\begin{equation}\label{eqn:hessian}
\nabla^{2,\check g(0)}d_{\check g(0)}(x,x_1) \leq C_8
\end{equation}
in the sense of barrier, outside $B_{\check g(0)}(x_1,1/2)$ by Hessian comparison. Take $\phi$ be a smooth non-increasing function on $[0,+\infty)$ such that $\phi=1$ on $[0,1]$, vanishes outside $[0,2]$ and satisfies $\phi''\geq -10^4\phi$, $|\phi'|^2\leq 10^4\phi$. Thanks to \eqref{eqn:hessian}  and curvature estimates, for some $C_9>0$ the cut-off function $\Phi(x)=e^{-C_9t}\phi(d_{\check g(0)}(x,x_1))$ satisfies 
\begin{equation} 
\left(\frac{\partial}{\partial t}-(\check g+\check h)^{ij}\check\nabla_i\check\nabla_j\right)\Phi \leq 0\;\; \text{and}\;\; \frac{|\check\nabla \Phi|^2}{\Phi}\leq C_{10}.
\end{equation}

Clearly, $t\check G \Phi=0$ at $t=0$. It suffices to consider the interior maximum in which
\begin{equation}
\begin{split}
0&\leq t\Phi\cdot \left(\frac{\partial}{\partial t}-(\check g+\check h)^{ij}\check\nabla_i\check\nabla_j\right)(t\check G \cdot \Phi)\\
&\leq t\check G \cdot \Phi+t^2\Phi^2 \left(C_6t^{-2} -\frac{ \check G^2}{(1+\Lambda_0)^2} \right)+2C_{10} t^2 \check G \Phi\\
&\leq C_6-C_7^{-1} t^2 \Phi^2 \check G^2+C_{11}t\check G \Phi
\end{split}
\end{equation}
for $t\in [0,1]$, where we have used $\nabla (\check G\Phi)=0$ at the interior maximum. This particularly implies $t\check G\Phi \leq C_{12}$ at the interior maximum and hence $\check G(x_1,1)\leq C_{13}$. Rescaling it back yields the result at $(x_1,t_1)$, this proves the case of $k=1$.   The higher order case can be done using similar but simpler argument.

\end{proof}

\subsection{Localized Maximum Principle}

We will need a local maximum principle,  developed by the author and Tam \cite{LeeTam2022}. 

\begin{prop}[Theorem 1.1 in \cite{LeeTam2022}]\label{prop:localMP}
Suppose $(M,g(t)),t\in [0,T]$ is a smooth solution to the Ricci flow which is possibly incomplete, such that $\Ric(g(t))\leq \a t^{-1}$ on $M\times (0,T]$ for some $\a>0$. Let $\varphi(x,t)$ be a continuous function on $M\times [0,T]$ such that $\varphi(x,t)\leq \a t^{-1}$ and  
$$\left(\frac{\partial}{\partial t}-\Delta_{g(t)}\right)\Big|_{(x_0,t_0)}\varphi\leq L\varphi$$
whenever $\varphi(x_0,t_0)>0$ in the sense of barrier, for some continuous function $L$ on $M\times [0,T]$ with $L\leq \a t^{-1}$. If $x_0\in M$ is a point so that $B_{g(0)}(x_0,2)\Subset M$ and $\varphi(x,0)\leq 0$ on $B_{g(0)}(x_0,2)$. Then for any $\ell\in\mathbb{N}$, there exists $\hat T_2(n,\ell,\a)>0$ such that for all $t\in [0,T\wedge \hat T_2]$,
$$\varphi(x_0,t)\leq t^\ell.$$
\end{prop}

As a direct Corollary, we use the maximum principle to see that Ricci-DeTurck flow is locally stable in $L^\infty$. This is crucial  in constructing local solution to Ricci-harmonic map heat flow.
\begin{cor}\label{cor:poly-growth-h}
There exists $ 10^{-2}>\hat \e_3(n)>0$ such that the following holds: Let $(N,\tilde g(t)),t\in [0,T]$ be a smooth solution to the Ricci flow which is not necessarily complete such that for some $\rho>0$, 
\begin{enumerate}
\item $B_{\tilde g(0)}(x_0,4\rho)\Subset N$ for some $x_0\in N$;
\item $|\Rm(\tilde g(t))|\leq \a t^{-1}$ for some $\a>0$ on $(0,T]$.
\end{enumerate}
If $\hat g(t)$ is a smooth solution to the Ricci-DeTurck flow with respect to $\tilde g(t)$ on $N\times [0,T]$ so that $\hat g(0)=\tilde g(0)$ and 
$$(1-\hat \e_3)\tilde g(t)\leq \hat g(t)\leq (1+\hat \e_3)\tilde g(t).$$
Then for all $\ell\in \mathbb{N}$, there exists $\hat T_3(n,\ell,\a)>0$ such that for all $(x,t)\in B_{\tilde g(0)}(x_0,\rho)\times [0,T\wedge \hat T_3\rho^2]$,
\begin{equation}
\left\{
\begin{array}{ll}
|\tilde g(t)-\hat g(t)|\leq \rho^{-2\ell}t^\ell;\\[2mm]
|\tilde\nabla \hat g(t)|^2\leq \rho^{-2\ell-2}t^\ell
\end{array}
\right.
\end{equation}
\end{cor}
\begin{proof}
By rescaling, we might assume $\rho=1$.
By Lemma~\ref{lma:RDF-higher-ord}, we might assume $|\tilde\nabla^2 h|\leq \hat L_2(n,\a)t^{-1}$ on $ B_{\tilde g(0)}(x_0,3)\times (0,T\wedge 1]$, provided that $ \hat \e_3$ is small enough. Since $h(0)=0$, the first inequality { on $B_{\tilde g(0)}(x_0,2)$} now follows from applying Proposition~\ref{prop:localMP} and Lemma~\ref{lma:improved-evo-h}.  The first order can be done by using Bernstein-Shi trick as in the proof of Lemma~\ref{lma:RDF-higher-ord}, using the improved estimate of $h$.   Alternatively by differentiating \eqref{eqn:RDF-3}, we have 
\begin{equation}
\begin{split}
\left(\frac{\partial}{\partial t}-\hat g^{ij}\tilde\nabla_i\tilde\nabla_j\right) \tilde\nabla h
&=\tilde\nabla h *h *\widetilde\Rm+h*\tilde\nabla\widetilde\Rm+\tilde\nabla h  *\widetilde\Rm\\
&\quad +   \tilde\nabla h * \tilde\nabla h *\tilde\nabla h+ \tilde\nabla^2 h* \tilde\nabla h
\end{split}
\end{equation}
where we have omitted the contraction of $\hat g$ and $\tilde g$, thanks to the metric equivalence.  Then we argue as in the proof of the second inequality in Lemma~\ref{lma:improved-evo-h} to {  show that on $B_{\tilde g(0)}(x_0,2)\times (0,T\wedge1]$,
\begin{equation}
\begin{split}
\left(\frac{\partial}{\partial t}-\Delta_{\tilde g(t)}\right) |\tilde\nabla h|^2
&\leq C_n |\tilde\nabla h|^2|\widetilde\Rm| +C_n |h||\tilde\nabla h||\tilde\nabla \widetilde\Rm|\\
&\quad +C_n|\tilde\nabla h|^4+ |\tilde\nabla h|^2 |\tilde\nabla^2 h|+C_n |h||\tilde\nabla h||\tilde\nabla^3 h|\\
&\leq C_n L |\tilde\nabla h|^2+t^{-2}|h|^2 
\end{split}
\end{equation}}
where $L:= t^{1/2}|\tilde\nabla\widetilde\Rm|+t^{1/2}|\tilde\nabla^3h|+ |\widetilde\Rm|+ |\tilde\nabla h|^2+|\tilde\nabla^2 h|$ is a smooth function satisfying $L\leq \b t^{-1}$ for some $\b=\b(n,\a)>0$, thanks to Shi's estimates and Lemma~\ref{lma:RDF-higher-ord}. In particular, the function $ 
\varphi:=|\tilde\nabla h|^2+ t^{-3}|h|^2$
satisfies 
\begin{equation}
\left(\frac{\partial}{\partial t}-\Delta_{\tilde g(t)}\right)\varphi\leq C_nL \varphi
\end{equation}
by Lemma~\ref{lma:improved-evo-h}. Furthermore $\varphi(0)=0$, thanks to the improved estimate of $h$. Result follows by applying Proposition~\ref{prop:localMP} to $\varphi$.

\end{proof}

\section{Existence of Ricci-harmonic map heat flow}\label{sec:RHMF}

In this section, we will construct local solution to the Ricci-harmonic map heat flow with quantitative estimates, depending only on the scaling invariant assumptions. This is based on idea of Hochard \cite{HochardPhd} and Simon-Topping \cite{SimonTopping2021} in producing Ricci flow with scaling invariant estimate.

\begin{thm}\label{thm:local-existence}
Suppose $M$ is a smooth manifold and $g(t),\tilde g(t),t\in [0,T]$ are two smooth solution to complete Ricci flow such that $g(0)=\tilde g(0)=g_0$ and for some $\a>0$, 
\begin{enumerate}
\item[(a)] $B_{g_0}(x_0,4)\Subset M$;
\item[(b)] $|\Rm(g(t))|+|\Rm(\tilde g(t))|\leq \a t^{-1}$.
\end{enumerate}
Then there exists $\hat T(n,\a)>0$ and a smooth solution 
$$F:B_{g_0}(x_0,1)\times [0,T\wedge \hat T]\to M$$
to \eqref{eqn:RHMHF} such that:
\begin{enumerate}
\item $F(x,0)=x$ for $x\in B_{g_0}(x_0,1)$;
\item  $F$ is a diffeomorphism onto its image.
\end{enumerate}
Furthermore, there is $C_0(n,\a)>0$ such that  for all $t\in [0,T\wedge\hat  T]$,  
$$(1-C_0t)g(t)\leq F_t^*\tilde g(t)\leq (1+C_0t) g(t)$$
and $d_{g_0}\left(x,F_t(x)\right)\leq C_0\sqrt{t}$.
\end{thm}
\begin{proof}

By Shi's estimate \cite{Shi1989}, we will also assume 
\begin{equation}
|\nabla \Rm(g(t))|+|\nabla\Rm(\tilde g(t))|\leq \a^{3/2}t^{-3/2}
\end{equation}
on $M\times (0,T]$. We will assume $T<1$ and $\a>1$ for convenience.

We start with specifying the constants that will be used in the construction. 
\begin{enumerate}
\item[(i)]  $\e_3(n)=\frac1{10}\hat \e_3(n)$ where $\hat \e_3$ is from Corollary~\ref{cor:poly-growth-h};
\item[(ii)] $\b(n)$ from Lemma~\ref{lma:l-balls};
\item[(iii)]  $\tilde T_2(n,10n)$ from Proposition~\ref{prop:extension-HM};
\item[(iv)]  $\hat T_3(n,\ell,\a)$ from Corollary~\ref{cor:poly-growth-h} where $\ell= 10[\a]+2$;
\item [(v)] $L_3(n,n,\a,10^2n,3)$ and $S_3(n,n,\a,10^2n,3)$ from Proposition~\ref{prop:HF-boot};
\item[(vi)]  $\hat L_4(n,\a)= 10^3n\a+ 10n\sqrt{L_3} $;
\item[(vii)] $\mu(n,\a)=10^{-2}\cdot \inf\{\tilde T_2 \a^{-1},1,(\hat L_4)^{-1}\e_3, (4\a)^{-1}\}$;
\item[(viii)] $\Lambda(n,\a):=10^3\cdot \sup\{\mu \sqrt{\frac{ L_3}{1-3\e_3} (1+\mu)^\a}, 1+\b \a^{1/2},\sqrt{S_3^{-1}},\sqrt{\hat T_3^{-1}},\sqrt{ \e_3^{-1}},\sqrt{L_3}+1 +  \b\sqrt{\a}\}$.
\end{enumerate}

\medskip

Since $B_{g_0}(x_0,4)\Subset M$, we  consider Dirichlet problem 
\begin{equation}
\left\{
\begin{array}{ll}
\partial_t F=\Delta_{g(t),\tilde g(t)} F\;\;\quad &\text{for}\; x\in B_{g_0}(x_0,4), \; t>0;\\[2mm]
F(x,0)=x \;\;\quad &\text{for}\; x\in B_{g_0}(x_0,4);\\[2mm]
F(x,t)=x \;\;\quad &\text{for}\; x\in \partial B_{g_0}(x_0,4).
\end{array}
\right.
\end{equation}
{ By approximation, we might assume $\partial B_{g_0}(x_0,4)$ is smooth.}
This has a solution for a short-time on $[0,t_0]$. By smoothness, we can assume $t_0$ is small enough so that when we restrict on $B_{g_0}(x_0,3)$, we have $F(x,0)=x$ and for all $t\in [0,t_0]$,
\begin{enumerate}
\item[(i)]$F_t|_{B_{g_0}(x_0,3)}$ is a diffeomorphism onto its image;
\item[(ii)]$(1-\e_3)g(t)\leq F_t^* \tilde g(t)\leq (1+\e_3)g(t)$ on $B_{g_0}(x_0,3)$.
\end{enumerate}  

We remark here that $t_0$ possibly depends on the geometry of $g(t)$, $\tilde g(t)$ on the compact set $B_{g_0}(x_0,4)$. We want to remove the dependence using idea of Hochard \cite{HochardPhd}.

\medskip
We define a sequence of real number inductively:
\begin{itemize}
\item $r_0=3$;
\item $r_{i+1}=r_i-6\Lambda\sqrt{t_i}$;
\item $t_{i+1}=(1+\mu) t_i$.
\end{itemize}
\medskip

We let $\mathcal{P}(k)$ be the following statement: There exists a smooth solution $F$ to \eqref{eqn:RHMHF} on $B_{g_0}(x_0,r_k)$ for $t\in [0,t_k]$ such that 
\begin{enumerate}
\item[(I)] $F_t|_{B_{g_0}(x_0,r_k)}$ is a  diffeomorphism onto its image;
\item[(II)] $(1-\e_3)g(t)\leq F_t^* \tilde g(t)\leq (1+\e_3) g(t)$ on $B_{g_0}(x_0,r_k)\times [0,t_k]$;
\item[(III)] For all $(x,t)\in B_{g_0}(x_0,r_k-5\Lambda\sqrt{t_k})\times [0,t_k]$, we have 
\begin{equation*}
\left\{
\begin{array}{ll}
d_{g_0}\left( x, F_t(x)\right)\leq \frac14 \Lambda\sqrt{t_k};\\[2mm]
B_{g_0}(x,\Lambda\sqrt{t_k})\subseteq F_t\left( B_{g_0}(x,2\Lambda\sqrt{t_k})\right).
\end{array}
\right.
\end{equation*}

\end{enumerate}
\medskip

From above discussion and smoothness, $\mathcal{P}(0)$ is true by shrinking $t_0$ if necessary. We claim that $\mathcal{P}(k)$ is true if $r_k>0$ and $t_k<T$. We will prove it by induction on $k$.

We assume $r_{k+1}>0$, $t_{k+1}<T$ and  $\mathcal{P}(k)$ holds. That said, we have a smooth solution $F$ to \eqref{eqn:RHMHF} on $B_{g_0}(x_0,r_k)\times [0,t_k]$ so that (I),  (II) and (III) above are true.  We first extend $F$ to a longer time interval on a smaller set.
\begin{claim}\label{claim:existence-F-extended}
There exists a smooth solution $F$ to \eqref{eqn:RHMHF} on $B_{g_0}(x_0,r_k-\Lambda\sqrt{t_k})\times [0,t_{k+1}]$ such that $|dF|^2\leq 10^2n$.
\end{claim}
\begin{proof}[Proof of Claim]
We consider the translated solution $g(t_k+t)$ and $\tilde g(t_k+t)$ for $t\in [0,\mu t_k]$. We apply Proposition~\ref{prop:extension-HM} on the translated Ricci flows with $U=B_{g_0}(x_0,r_k)$, $\rho=\sqrt{\frac{t_k}{\a}}$, $\Lambda_0=10n$ and $f=F_{t_k}$ to obtain a smooth solution $\hat F$ on $U\times [0,\mu t_k\wedge \tilde T_2 \rho^2]=U\times [0,\mu t_k]$ (thanks to choice of $\mu$) such that $\hat F(0)=F_{t_k}$ and $|d\hat F|^2\leq 10^2n$. Moreover, $\hat F$ solves Ricci-harmonic map heat flow with respect to $g(t)$ and $\tilde g(t)$ on $U_\rho=\{x\in U: B_{g(t_k)}(x,\rho)\Subset U\}$ for $t\in [0,\mu t_k]$. Thanks to  Lemma~\ref{lma:l-balls}, for $x\in B_{g_0}(x_0,r_k-\Lambda\sqrt{t_k})$ we have
\begin{equation}
\begin{split}
B_{g(t_k)}\left(x,\sqrt{\frac{t_k}{\a}}\right)&\subseteq B_{g_0}\left(x,(\a^{-1/2}+\b \a^{1/2}) \sqrt{t_k}\right)
\Subset U,
\end{split}
\end{equation}
thanks to choice of $\Lambda$. Therefore, if we extend $F$ on $B_{g_0}(x_0,r_k-\Lambda\sqrt{t_k})$ by defining 
\begin{equation}
F(x,t)=\left\{
\begin{array}{ll}
F(x,t),\;\;\text{if}\;\; t\in [0,t_k];\\[2mm]
\hat F(x,t-t_k),\;\;\text{if}\;\; t\in [t_k,t_{k+1}],
\end{array}
\right.
\end{equation}
then $F$ is a solution to \eqref{eqn:RHMHF} on $B_{g_0}(x_0,r_k-\Lambda\sqrt{t_k})\times [0,t_{k+1}]$. 
\end{proof}

We now show that $F$ remains  a local diffeomorphism when restricted on a smaller set, by showing a weaker form of (II) in $\mathcal{P}(k+1)$.
\begin{claim}\label{claim:local-diff}
The smooth solution $F$ obtained from Claim~\ref{claim:existence-F-extended} satisfies
$$(1-3\e_3)g(t)\leq F_t^* \tilde g(t)\leq (1+3\e_3)g(t)$$
on $B_{g_0}(x_0,r_k-2\Lambda\sqrt{t_k})\times [0,t_{k+1}]$. In particular, $F_t|_{B_{g_0}(x_0,r_k-2\Lambda\sqrt{t_k})}$ is a local diffeomorphism for $t\in [0,t_{k+1}]$.
\end{claim}
\begin{proof}[Proof of claim]

Let  $x_2\in B_{g_0}(x_0,r_k-2\Lambda\sqrt{t_k})$. By applying Proposition~\ref{prop:HF-boot} to $F$ on $B_{g_0}(x_2,\Lambda\sqrt{t_k})\times [0,t_{k+1}]$ with $\rho=\frac14\Lambda\sqrt{t_k}$, we conclude that at $x_2$, for all $ 0\leq m \leq 3$ and $t\in [0,t_{k+1}\wedge S_3\rho^2]=[0,t_{k+1}]$ (thanks to our choice of $\Lambda$ and $\mu$), we have 
\begin{equation}\label{eqn:esti-F}
 |D^m dF|^2 \leq L_3 t^{-m}.
\end{equation}

We now establish the metric equivalence using \eqref{eqn:esti-F}. The argument is point-wise. Fix $x\in B_{g_0}(x_0,r_k-2\Lambda\sqrt{t_k})$. We only show the lower bound, the upper bound is identical.  By induction hypothesis, $F_t^*\tilde g(t)\geq (1-\e_3)g(t)$ for $t\in [0,t_k]$, at $x$. Let $\phi(t)$ be a smooth non-negative function such that $\phi(t)\leq 1$ and $h(t):=F_t^*\tilde g(t)-\phi(t)g(t)> 0$ at $t=t_k$. We want to show that $h(t)>0$ for all $t\in [t_k,t_{k+1})$. If this is not the case, then there is $s_0\in (t_k,t_{k+1})$ such that $\mathrm{Null}\left(h(s_0)\right)\neq \emptyset$ and $h(t)>0$ for all $t\in [t_k,s_0)$. Let $v$ be the unit vector with respect to $g(s_0)$ such that  $h(s_0)(v,v)=0$. Then at $t=s_0$,  \eqref{eqn:RHMHF},
\begin{equation}
\begin{split}
 0&\geq \partial_t \left(h(t)(v,v) \right)|_{ t=s_0}\\
 &=2v^iv^j\left(\tilde g_{\a\b}F^\a_i F^\b_{tj}-F^\a_i F^\b_j\tilde R_{\a\b}  \right)-\phi'+2\phi R_{vv}\\
 &\geq -\phi' - t^{-1}\left(10^3n\a+ 10n\sqrt{L_3} \right) >-\phi'-\hat L_4 t_k^{-1},
\end{split}
\end{equation}
using our choice of $\hat L_4$. If we choose $\phi(t)=1-2\e_3-\hat L_4 t_k^{-1}(t-t_k)$, then we see that $s_0$ does not exist and thus,
\begin{equation}
F_t^*\tilde g(t)\geq (1-3\e_3)  g(t)
\end{equation}
on $B_{g_0}(x_0,r_k-2\Lambda\sqrt{t_k})\times [0,t_{k+1}]$, by induction hypothesis and our choice of $\Lambda$ and $\mu$. The upper bound of $F_t^*\tilde g(t)$ is similar. This proves the claim.
\end{proof}

\medskip

Next, we show that (I) in  $\mathcal{P}(k+1)$ holds. 

\begin{claim}\label{claim:diff-local}
Let $F_t$ be the one parameter family of smooth map obtained from Claim~\ref{claim:existence-F-extended}, then $F_t|_{B_{g_0}(x_0,r_k-3\Lambda\sqrt{t_k})}$ is a diffeomorphism onto its image for $t\in [0,t_{k+1}]$. 
\end{claim}
\begin{proof}
By induction hypothesis and Claim~\ref{claim:local-diff}, it suffices to show that if $x_1,x_2\in B_{g_0}(x_0,r_k-3\Lambda\sqrt{t_k})$ and $t_0'\in (t_k,t_{k+1}]$ such that $F_{t_0'}(x_1)=F_{t_0'}(x_2)$, then $ x_1=x_2$. Suppose $x_1\neq x_2$, we can assume $t_0'$ to be the first $t_0'\in (t_k,t_{k+1}]$ such that $w:=F_{t_0'}(x_1)=F_{t_0'}(x_2)$. 
 
We consider the set
$$\Sigma:=\{(z,t)\in B_{g_0}(x_0,r_k-2\Lambda\sqrt{t_k})\times [t_k,t_0']: F_t(z)=w\}.$$
Thanks to Claim~\ref{claim:local-diff}, the map $\Phi:B_{g_0}(x_0,r_k-2\Lambda\sqrt{t_k})\times [t_k,t_0'] \to M$ given by $\Phi(x,t):=F_t(x)$ is of full rank and hence $\Sigma$ is one dimensional manifold traversal to { constant-time slices} such that $(x_1,t_0'),(x_2,t_0')\in \Sigma$. We let $x_i(t)$ be time curve such that $x_i(t_0')=x_i$ and $\Phi(x_i(t),t)\equiv w$. We claim that $x_i(t)$ lies in $B_{g_0}(x_0,r_k-2\Lambda\sqrt{t_k})$ for all $t\in [t_k,t_0']$. If this is true, then $F_{t_k}\left(x_1(t_k)\right)=F_{t_k}\left(x_2(t_k)\right)$ so that $x_1(t_k)=x_2(t_k)$ by induction hypothesis. This forces $x_1(t)\equiv x_2(t)$ for $t\in [t_k,t_0']$. 

It remains to control the location of $x_i(t)$. By differentiating $\Phi(x_i(t),t)\equiv w$, we have $
 F^*\tilde g(x_i',x_i')=\tilde g(F_t,F_t)$ so that Claim~\ref{claim:local-diff} and \eqref{eqn:esti-F} imply that  
\begin{equation}\label{eqn:x'-speed}
(1-3\e_3)|x_i'|_{g(t_k)}^2\leq L_3t^{-1}(1+\mu)^\a.
\end{equation}
on $[t_{min},t_0']$ where $t_{min}\in [t_k,t_0']$ is minimal of time such that $x_i(t)\in B_{g_0}(x_0,r_k-2\Lambda\sqrt{t_k})$ for $t\in [t_{min},t_0']$. In particular,
\begin{equation}\begin{split}
d_{g_0}(x_i(t_{min}),x_i(t_{0}'))&\leq d_{g(t_k)}(x_i(t_{min}),x_i(t_{0}'))+\b\sqrt{\a t_k}\\
&\leq \int^{t_0'}_{t_{min}} |x_i'|_{g(t_k)} ds +\b\sqrt{\a t_k}\\
&\leq \left(\mu \sqrt{\frac{ L_3}{1-3\e_3} (1+\mu)^\a}+\b\sqrt{\a} \right) \sqrt{t_k}\\
&<\Lambda \sqrt{t_k}
\end{split}
\end{equation}
by Lemma~\ref{lma:l-balls}, \eqref{eqn:x'-speed} and our choice of $\Lambda$. Therefore, $
x_i(t_{min})\in B_{g_0}(x_0,r_k-2\Lambda\sqrt{t_k})$ and thus $t_{min}=t_k$.

\end{proof}

Finally, we intend to improve the estimates on the smaller domain.
\begin{claim}\label{claim:dist-II-III}
The smooth map $F$ obtained by Claim~\ref{claim:existence-F-extended} satisfies (II) and (III) in $\mathcal{P}(k+1)$. 
\end{claim}
\begin{proof}
We first improve the asymptotic of $\partial_t F$ at $t=0$. Fix $x_1\in B_{g_0}(x_0,r_k-5\Lambda\sqrt{t_k})$. Using induction hypothesis (III) in $\mathcal{P}(k)$, we can consider the flow $\hat g(t):=(F_t^{-1})^*g(t)$ on $B_{g_0}(x_1,\Lambda\sqrt{t_k})\times [0,t_k]$. By (II) in $\mathcal{P}(k)$, we can apply Corollary~\ref{cor:poly-growth-h} with $\rho=\frac14 \Lambda \sqrt{t_k}$ to conclude that for all $t\in [0,t_k\wedge (16^{-1}\hat T_3  \Lambda^2t_k)]=[0,t_k]$ (by our choice of $\Lambda$) and $x\in B_{g_0}(x_1,\frac14 \Lambda\sqrt{t_k})$,
\begin{equation}\label{eqn:esti-g-RDF}
\left\{
\begin{array}{ll}
|\tilde g(t)-\hat g(t)|\leq \left(16\Lambda^{-2} t_k^{-1} t\right)^{\ell} ;\\[2mm]
|\tilde\nabla \hat g(t)|^2\leq \left(16\Lambda^{-2} t_k^{-1} \right)^{\ell+1}t^\ell.
\end{array}
\right.
\end{equation}

Since $F_t(x_1)\in B_{g_0}(x_1,\frac14 \Lambda\sqrt{t_k})$, using our choice of $\Lambda$ and $\ell\geq 1$, \eqref{eqn:esti-g-RDF} implies 
\begin{equation}\label{eqn:metric-improved-RDF-0}
\left(1-(16\Lambda^{-2}t_k^{-1}t)^\ell\right) g(t)\leq F_t^*\tilde g(t)\leq \left(1+(16\Lambda^{-2}t_k^{-1}t)^\ell\right) g(t)
\end{equation}
and thus
 \begin{equation}\label{eqn:metric-improved-RDF}
\left(1-\frac13 \e_3\right) g(t)\leq F_t^*\tilde g(t)\leq \left(1+\frac13 \e_3\right) g(t)
\end{equation}
on $B_{g_0}(x_0,r_k-5\Lambda\sqrt{t_k})\times [0,t_k]$. Now (II) of $\mathcal{P}(k+1)$ follows from the argument in the Claim~\ref{claim:local-diff} with $\e_3$ replaced by $\frac13 \e_3$.  

\medskip
It remains to establish (III) of $\mathcal{P}(k+1)$. We need to use the first order estimate in \eqref{eqn:esti-g-RDF} to improve the asymptotic of $F_t$ as $t\to0$.  Fix $x_2\in  B_{g_0}(x_0,r_{k}-5\Lambda\sqrt{t_{k}})$,
where we know from \eqref{eqn:esti-g-RDF} that $ |\partial_t F_t(x_2)|^2\leq \left(16\Lambda^{-2} t_k^{-1} \right)^{\ell+1}t^\ell$ for $t\in [0,t_k]$. Consider the curve $\gamma(s)=F_s(x_2), s\in [0,t]$. By Lemma~\ref{lma:l-balls}, for $t\in [t_k,t_{k+1}]$
\begin{equation}
\begin{split}
d_{g_0}\left(x_2, F_t(x_2)\right)&\leq d_{\tilde g(t)}\left(x_2, F_t(x_2)\right)+\b\sqrt{\a t}\\
&\leq  \int^t_0 |\gamma'(s)|_{\tilde g(t)}\, ds  +\b\sqrt{\a t}\\
&\leq \left( \int^t_{t_k}+\int^{t_k}_0\right)|\partial_sF_s(x_2)|_{\tilde g(t)}\, ds  +\b\sqrt{\a t}\\
&=\mathbf{A}+\mathbf{B}+\b\sqrt{\a t}.
\end{split}
\end{equation}

Using the rough estimate from \eqref{eqn:esti-F} with $ m=1$, we have
\begin{equation} 
\begin{split}
\mathbf{A}&\leq 2\a(1+\mu)\int^{t_{k+1}}_{t_k} |\partial_s F_s(x_2)|_{\tilde g(s)}\,ds\\
&\leq 2\a\mu(1+\mu)\sqrt{L_3t_k} \leq \sqrt{L_3t_k}
\end{split}
\end{equation}
while we use the Ricci flow equation and the improved estimate of $\partial_t F_t(x_2)$ to deduce
\begin{equation}
\begin{split}
\mathbf{B}&\leq \int^{t_k}_0  t^\a s^{-\a}|\partial_sF_s(x_2)|_{\tilde g(s)}\, ds \\
&\leq t^\a  (16\Lambda^{-2} t_k^{-1})^{(\ell+1)/2}\int^{t_k}_0 s^{\frac12 \ell-\a} ds\\
&\leq  (16\Lambda^{-2} )^{(\ell+1)/2}  \sqrt{t_{k}}\leq \sqrt{t_{k+1}}
\end{split}
\end{equation}
where we have used $16\Lambda^{-2}\leq 1$  and choice of $\ell$.

Combining all, we conclude that for all $x_2\in B_{g_0}(x_0,r_{k}-5\Lambda\sqrt{t_{k}})$ and $t\in [t_k,t_{k+1}]$,
\begin{equation}\label{eqn:improved-dist-compare}
 d_{g_0}\left(x_2,F_t(x_2)\right)\leq \left( \sqrt{L_3}+1 +  \b\sqrt{\a}\right)\sqrt{t_{k+1}}\leq \frac14 \Lambda \sqrt{t_{k+1}}.
\end{equation}

Since $B_{g_0}(x_0,r_{k+1}-5\Lambda\sqrt{t_{k+1}})\subseteq B_{g_0}(x_0,r_{k}-5\Lambda\sqrt{t_{k}})$, this proves the first part of (III) in $\mathcal{P}(k+1)$. It remains to prove the second part of (III) in $\mathcal{P}(k+1)$. 
Fix $x_3\in B_{g_0}(x_0,r_{k+1}-5\Lambda\sqrt{t_{k+1}})$. Since $F_0=\mathrm{id}$, we let $s> 0$ be the maximal time in $[0,t_{k+1}]$ such that for all $t\in [0,s]$, 
$$ B_{g_0}(x_3,\Lambda\sqrt{t_{k+1}})\Subset F_t\left(B_{g_0}(x_3,2\Lambda\sqrt{t_{k+1}}) \right).$$

If $s<t_{k+1}$, then there is $z$ such that $d_{g_0}(x_3,z)=2\Lambda\sqrt{t_{k+1}}$ and $d_{g_0}\left(x_3,F_s(z)\right)=\Lambda \sqrt{t_{k+1}}$. In particular, $d_{g_0}(x_0,z)\leq r_{k}-5\Lambda\sqrt{t_k}$ so that \eqref{eqn:improved-dist-compare} applies  at $z$. Thus,
\begin{equation}
\begin{split}
2\Lambda\sqrt{t_{k+1}}=d_{g_0}(x_3,z)&\leq d_{g_0}(x_3,F_s(z))+d_{g_0}(F_s(z),z)\\
&\leq \Lambda\sqrt{t_{k+1}}+\frac14 \Lambda\sqrt{t_{k+1}}.
\end{split}
\end{equation}
This is impossible and hence $s=t_{k+1}$. This proves (III) of $\mathcal{P}(k+1)$.
\end{proof}

By induction, we have shown that $\mathcal{P}(k)$ is true if $t_k<T$ and $r_k>0$. We now claim that there is a smooth solution $F$ to \eqref{eqn:RHMHF} on $B_{g_0}(x_0,1)\times [0,T\wedge \hat T]$ for some uniform $\hat T(n,\a)>0$. Since $r_i\to-\infty$ as $i\to+\infty$, we can find $k_0\in\mathbb{N}$ such that $r_{k_0}>2\geq r_{k_0+1}$. We now estimate the corresponding $t_{k_0}$. 
\begin{equation}
\begin{split}
{ 1\leq }r_0-r_{k_0+1}&=-\sum_{i=0}^{k_0} (r_{i+1}-r_i)\\
&=6\Lambda \sum_{i=0}^{k_0} \sqrt{t_i}\\
&\leq 6\Lambda \sqrt{t_{k_0}}\cdot  \sum_{i=0}^{\infty} (1+\mu)^{-i/2}\\
&:={ \hat T^{-1/2}}\cdot \sqrt{t_{k_0}}
\end{split}
\end{equation}
where $\hat T=\hat T(n,\a)$. Hence, $t_{k_0}\geq \hat T$. We replace $T$ by $T\wedge \hat T$ and further insist $5\Lambda\sqrt{\hat T}\leq 1$. Hence, there must be another $k_0'<k_0$ so that $t_{k_0'}<T\wedge \hat T\leq t_{1+k_0'}$, $r_{k_0'}>r_{k_0}>2$ and $\mathcal{P}(k_0')$ is true. Thus we conclude the existence of $F$ on $B_{g_0}(x_0,2)\times [0,\hat T\wedge T]$, after replacing $\hat T$ by $(1+\mu)^{-1}\hat T$. Now the distance estimate follows from  Claim~\ref{claim:dist-II-III} with $k=k_0'$ and $r_{k_0'}-5\Lambda\sqrt{t_{k_0'}}\geq 1$. { Finally, we apply Corollary~\ref{cor:poly-growth-h} again so that}
\begin{equation}
(1-C_0t) g(t)\leq F_t^*\tilde g(t)\leq (1+C_0t) g(t)
\end{equation}
on $B_{g_0}(x_0,1)\times [0,T\wedge \hat T]$, for some $C_0(n,\a)>0$.  This completes the proof by restricting $F$ to $B_{g_0}(x_0,1)\times [0,T\wedge \hat T]$.
\end{proof}

\section{Uniqueness of Ricci flows} \label{sec:unique}
In this section, we use the local existence of Ricci-harmonic map heat flow to show uniqueness of Ricci flow with scaling invariant estimate.

\begin{proof}[Proof of Theorem~\ref{thm:unique}] 
For any { large } $R>1$, we consider $g_R(t)=R^{-2}g(R^2t)$ and $\tilde g_R(t)=R^{-2}\tilde g(R^2t)$ for $t\in [0,R^{-2}T]$. By Theorem~\ref{thm:local-existence}, there exists a smooth solution $\hat F_R(t)$ on $B_{g_{R,0}}(x_0,1)\times [0,TR^{-2}\wedge \hat T]$. By considering $F_R(x,t)=\hat  F_R(x,R^{-2}t)$, we obtain a sequence of Ricci-harmonic map heat flow $\{F_R\}_{R>1}$ defined on $B_{g_{0}}(x_0,R)\times [0,T]$ such that 
\begin{enumerate}
\item $F_R(x,0)=x$;
\item  $F_R$ is a diffeomorphism onto its image;
\item $(1-C_0R^{-2}t) g(t)\leq (F_R)_t^*\tilde g(t)\leq (1+C_0R^{-2}t) g(t)$;
\item $d_{g_0}\left(x,(F_R)_t(x)\right)\leq C_0\sqrt{t}$ 
\end{enumerate}
on $B_{g_0}(x_0,R)\times [0,T]$. The property (4) shows that $F_R$ maps compact set to compact set, uniform in $R\to+\infty$. Since $|dF_R|$ is uniformly bounded, by local parabolic Schauder estimate, $F_{R}$ sub-converges to a global solution to the Ricci-harmonic map heat flow $F:M\times [0,T]\to M$ as $R\to +\infty$ such that $F(0)=\mathrm{Id}$ and $F_t^*\tilde g(t)\equiv g(t)$ on $M\times[0,T]$. In particular, $F_t$ is a local isometry. By Lemma~\ref{lma:evo-energy}, we have $DdF\equiv 0$ and hence $\Delta F=\partial_t F=0$. This implies $F_t(x)=x$ for all $(x,t)\in M\times [0,T]$ and therefore $g(t)\equiv \tilde g(t)$ for $t\in [0,T]$. This completes the proof.
\end{proof}

\medskip

By Theorem~\ref{thm:unique}, we have uniqueness as long as complete Ricci flow has bounded curvature instantaneously in scaling invariant way. Particularly, this is the case when the initial three-fold is uniformly non-collapsed with $\Ric\geq 0$.
\begin{proof}[Proof of Corollary~\ref{cor:strong-unique}] 
By the work of Simon-Topping \cite{SimonTopping2021}, there is a short-time solution to the  Ricci flow $\tilde g(t),t\in [0,T]$ starting from $g_0$ with $|\Rm(\tilde g(t))|\leq \a t^{-1}$, see also \cite{Lai2019}. It remains to show that any complete solution coincide with $\tilde g(t)$ for small $t$. Let $g(t),t\in [0,\tilde T]$ be a complete solution to the Ricci flow starting from $g_0$. By \cite[Theorem 1.2]{ChenXuZhang2013}, $\Ric(g(t))\geq 0$ for all $t\in [0,\tilde T]$. It follows from \cite[Lemma 2.1]{SimonTopping2023} that for some $\hat T(n,v_0)>0$, we have $|\Rm(g(t))|\leq \a t^{-1}$ on $(0,\tilde T\wedge \hat T]$ for some $\a(n,v_0)>0$. By Theorem~\ref{thm:unique}, $g(t)\equiv \tilde g(t)$ for $t\in [0,T\wedge \tilde T\wedge \hat T]$.
\end{proof}

Using the same method, we have the same result for complete $U(n)$ invariant \KR flow.

\begin{cor}\label{cor:strong-unique-KRF}
Suppose $(\mathbb{C}^n,g_0)$ is a complete \K metric with $U(n)$ symmetry such that 
\begin{enumerate}
\item $\mathrm{BK}(g_0)\geq 0$;
\item $ \inf_{\mathbb{C}^n}\mathrm{Vol}_{g_0}\left(B_{g_0}(x,1) \right)>0$,
\end{enumerate}
then any complete solution to $U(n)$ invariant \KR flow coincides with the solution constructed by Chau-Li-Tam \cite{ChauLiTam2017}, for a short-time.
\end{cor}
\begin{proof}
The proof is identical to Corollary~\ref{cor:strong-unique}. The $U(n)$ invariant \KR flow has been constructed by \cite[Corollary 3.1]{ChauLiTam2017}. It remains to show that under non-collapsing assumption, all complete solution with $U(n)$ symmetry coincides with their solution.

By \cite[Theorem 3.1]{YangZhang2013}, the non-negativity of bisectional curvature is preserved under complete \KR flow with $U(n)$ symmetry. It then follows from \cite[Lemma 3.1]{LeeTamPAMS}  that the \KR flow has scaling invariant curvature decay. The uniqueness then follows from Theorem~\ref{thm:unique}.

\end{proof}

\end{document}